\documentclass[12pt,reqno]{amsart}

\usepackage{amsmath, amsthm, amssymb}
\usepackage{amsfonts}
\usepackage{bm,mathrsfs}
\usepackage{enumitem}
\usepackage{hyperref}
\usepackage[dvips]{epsfig}
\usepackage{graphicx}
\usepackage[english]{babel}
\usepackage{mathrsfs}
\usepackage{thmtools}
\theoremstyle{plain}

\usepackage[backgroundcolor=white, bordercolor=blue,
linecolor=blue]{todonotes}

\usepackage{dsfont}
\usepackage{bbm}

\newtheorem{theorem}{Theorem}[section]

\newtheorem{lem}{Lemma}[section]
\newtheorem{prop}{Proposition}[section]

\hypersetup{
    colorlinks=true,
    citecolor=green,
    filecolor=green,
    linkcolor=blue,
    urlcolor=black
}

\numberwithin{equation}{section}

\begin{document}

\title[\tiny{$L_p$ Brunn-Minkowski inequality for eigenvalue of generalized Monge-Ampère equation}]{$L_p$ Brunn-Minkowski inequality for the eigenvalue of generalized Monge-Ampère equation}

\author{Yuxiao Yan}
\author{Feida Jiang$^*$}

\address{School of Mathematics and Shing-Tung Yau Center of Southeast University, Southeast University, Nanjing 211189, P.R. China}
\email{\url{yan_yuxiao@seu.edu.cn}}

\address{School of Mathematics and Shing-Tung Yau Center of Southeast University, Southeast University, Nanjing 211189, P.R. China; \newblock Shanghai Institute for Mathematics and Interdisciplinary Sciences Shanghai 200433, P. R. China}
\email{\url{jiangfeida@seu.edu.cn}}

\date{\today}
	\thanks{*corresponding author}



\begin{abstract}
We investigate the eigenvalue of a generalized Monge-Ampère equation. We proved that the eigenvalue of the equation on the $L_p$ Minkowksi addition of two convex bodies satisfies the $L_p$ Brunn-Minkowski inequality.
\end{abstract}

\maketitle

\section{Introduction}
The eigenvalue problem for the classical Monge-Ampère operator on smooth, bounded and uniformly convex domains $\Omega$ in $\mathbb{R}^n$ was first investigated by Lions \cite{Lions}. He considered the eigenvalue problem:
\begin{equation}\label{classical MA eigen}
\begin{cases}
\det D^2u = \lambda|u|^n &\text{ in }\Omega\\
u = 0 &\text{ on }\partial\Omega,
\end{cases}
\end{equation}
and showed the existence of the unique eigenvalue and eigenfunctions for \eqref{classical MA eigen}. The eigenfunctions are unique up to positive multiplicative constants. Tso \cite{Tso} used the variational approach introduced by Bakelman \cite{Bakelman1,Bakelman2} to study such problems. He gave a characterization of the Monge-Ampère eigenvalue for sufficiently smooth, bounded and uniformly convex domain $\Omega$:
\begin{equation}\label{classical variational character}
\lambda(\Omega)=\inf\left\{ \frac{\int_{\Omega} (-u)\det D^2 u~dx}{\int_{\Omega}(-u)^{n+1}~dx}: u\in C^{0,1}(\overline{\Omega})\cap C^2(\Omega)~\text{is nonzero, convex in }
 \Omega,~u\vert_{\partial\Omega}=0\right\}.
\end{equation}

Based on \eqref{classical variational character}, Salani \cite{Salani} proved a Brunn-Minkowski type inequality for the Monge-Ampère eigenvalue. For $C_+^2$ uniformly convex bounded domains $\Omega_0$ and $\Omega_1$, he showed that
\begin{equation}\label{BM}
\lambda(\Omega_\alpha)^{-\frac1 {2n}}\ge (1-\alpha)\lambda(\Omega_0)^{-\frac1{2n}}+\alpha \lambda(\Omega_1)^{-\frac1{2n}}, 
\end{equation}
where $\alpha\in[0,1]$, and $\Omega_\alpha=(1-\alpha) \Omega_0+\alpha \Omega_1$ is the Minkowski addition of two convex domains. Salani conjectured in his paper that the above inequality also holds for general convex domains. This conjecture was completely solved by Le \cite{Le2,Le1}. Le proved that for all bounded convex domains the Brunn-Minkowski inequality \eqref{BM} is true, and equality in \eqref{BM} holds if and only if the convex bodies $\Omega_0$ and $\Omega_1$ are homothetic.

Tong and Yau \cite{Tong-Yau} considered the eigenvalue problem for a generalized class of Monge-Ampère equations:
\begin{equation}\label{general MA eigen}
\begin{cases}
(u^\star)^k\det D^2u = \lambda|u|^{n+k} &\text{ in }\Omega\\
u = 0 &\text{ on }\partial\Omega,
\end{cases}
\end{equation}
where $k\in\mathbb{R}$ could be any real number with $n+k$ nonnegative, $\Omega\subset \mathbb{R}^n$ is a bounded strictly convex domain containing the origin as its interior point, and the function
\[
u^\star(x)=x\cdot Du(x)-u(x)
\]
is the Legendre transform of $u$ evaluated at the point $Du(x)$. The equation \eqref{general MA eigen} arises from the study of affine spheres \cite{Chen-Huang,Klartag}, and is related to the construction of complete Calabi-Yau metrics \cite{Collins-Li}. Tong and Yau discovered a class of variational functionals of \eqref{general MA eigen}
\begin{equation*}
    H_{n+k}(u)=\frac1{n+k+1}\int_{\Omega}(-u)(u^\star)^k~d\mu_u.
\end{equation*}
Here $d\mu_u$ denotes the Monge-Ampère measure, which can be written as $\det D^2u~dx$ if $u$ is $C^2$.
They also showed that there exists a unique eigenvalue and a unique (up to positive multiplicative constants) eigenfunction solving the problem \eqref{general MA eigen}. Naturally, the variational structure of the generalized Monge-Ampère operator $(u^\star)^k\det D^2u$ gives a characterization of the eigenvalue:
\begin{equation}\label{general variational character}
\lambda(\Omega)=\inf\left\{ \frac{\int_{\Omega} (-u)(u^\star)^k~d\mu_u}{\int_{\Omega}(-u)^{n+k+1}~dx}: u\in C^{1}(\overline{\Omega})~\text{is strictly convex in }
 \Omega,~u\vert_{\partial\Omega}=0\right\}.
\end{equation}
Moreover, the eigenfunction is the minimizer of the functional in \eqref{general variational character}. From now on, the notation $\lambda(\Omega)$ always denotes the eigenvalue of \eqref{general MA eigen} on the domain $\Omega$.

Our main result is the following theorem.
\begin{theorem}[$L_p$ Brunn-Minkowski inequality]\label{main result}
    Suppose $k\ge 0$. Let $\Omega_0$ and $\Omega_1$ be bounded, smooth and strictly convex domains containing the origin as an interior point, and $\Omega_\alpha:=(1-\alpha)\Omega_0+_p \alpha\Omega_1$ be the $L_p$-Minkowski addition of $\Omega_0$ and $\Omega_1$, where $p:=k+1\ge 1$ and $\alpha\in[0,1]$. Then
\begin{equation}\label{Lp BM}
    \lambda(\Omega_\alpha)^{-\frac{p}{2n}}\ge (1-\alpha)\lambda(\Omega_0)^{-\frac{p}{2n}}+\alpha\lambda(\Omega_1)^{-\frac{p}{2n}}.
\end{equation}
\end{theorem}

When $p=1$, the $L_p$-Minkowski addition $\Omega_\alpha$ degenerates to the classical Minkowski addition of $\Omega_0$ and $\Omega_1$. Therefore, \eqref{Lp BM} is a generalization of Salani's Brunn-Minkowski inequality \eqref{BM}.

The method that Salani used in his work \cite{Salani} to prove the classical Brunn-Minkowski for Monge-Ampère eigenvalue is the infimum convolution of two convex functions and the Prékopa–Leindler inequality \cite[Theorem~6.4]{Villani}. Simlimarly, our proof of \eqref{Lp BM} also relies on two key technical ingredients. On the one hand, we define a new structure called $L_p$ convolution, which is a generalization of the infimum convolution and compatible with the equation \eqref{general MA eigen}. On the other hand, we need an $L_p$ version of the Prékopa–Leindler inequality, which can be found in \cite{Wu}.

This paper is organized as follows. In section \ref{pre}, we will recall some basic concepts in convex geometry, Monge-Ampère equation, and provide necessary technical inequalities. In section \ref{Lp type convolution}, we introduce the $L_p$ convolution of two convex functions, and prove an identity for generalized Monge-Ampère energy. Finally, in section \ref{proof of main theorem}, we will apply the $L_p$ version of the Prékopa–Leindler inequality to prove Theorem \ref{main result}.

\section{Preliminaries}\label{pre}

\subsection{$L_p$ Minkowski addition of convex bodies}
Let us first recall some basic definitions in convex geometry. Let $\mathcal{K}^n$ be the class of convex bodies in $\mathbb{R}^n$. For $K\in\mathcal{K}^n$, its support function is defined to be
\[
h_K(u)=\sup\limits_{x\in K}~x\cdot u, \quad u\in\mathbb{R}^{n}.
\]
It is easy to check that the support function $h_K$ is a $1$-homogeneous, subadditive and convex function. Conversely, given a $1$-homogeneous, subadditive and convex function, there exists a unique convex body whose support function is the given one. In other words, a convex body and its support function uniquely determine each other.

Recall the Minkowski addition of two convex bodies $K,L\in\mathcal{K}^n$ is 
\[
K+L=\{x+y:x\in K,y\in L\},
\]
which is also a convex body. A significant fact about support function is Minkowski additivity:
\[
h_{K+L}(u)=h_{K}(u)+h_L(u).
\]
In 1960s, Firey \cite{Firey} generalized the Minkowski addition by defining the $L_p$ Minkowski addition $(1-\alpha)K+_p \alpha L$. When $p>1$, it is defined as the convex body with support function 
\[
h_{K,L,p}=((1-\alpha)h_K^p+\alpha h_L^p)^{1/p}. 
\]
However, when $p\in(0,1)$, the function $h_{K,L,p}$ may not be a support function for any convex body in general. Böröczky et al. \cite{Boroczky-Lutwak-Yang-Zhang} gave a feasible way to define $(1-\alpha)K+_p \alpha L$ to be a convex body:
\[
(1-\alpha)K+_p \alpha L:=\bigcap_{u\in\mathbb{S}^{n-1}}\{x\in\mathbb{R}^n:x\cdot u\le [(1-\alpha)h_K^p(u)+\alpha h_L^p(u)]^{\frac1p}\}.
\]
Note that the case $p=1$ is exactly the classical Minkowski addition.

\subsection{Legendre transform}
For the sake of rigor, we restate the definition of convex functions \cite{Schneider}. We say $u:\mathbb{R}^n\to \mathbb{R}\cup\{+\infty\}$ is convex, if $\{u=+\infty\}\ne\mathbb{R}^n$ and
\[
u((1-\lambda)x+\lambda y)\le (1-\lambda)u(x)+\lambda u(y),\quad \lambda\in[0,1].
\]
The set $\mathrm{dom}~u:=\{u\ne+\infty\}$ is called the effective domain of $u$. When we say a convex function is defined on $D\subset \mathbb{R}^n$, it means we extend $u$ to the whole space by setting $u=+\infty$ outside $D$.

Now we define the Legendre transform of a convex function $u:\mathbb{R}^n\to \mathbb{R}\cup\{+\infty\}$ to be
\[
u^*(\xi):=\sup_{x\in\mathbb{R}^n}\{x\cdot\xi-u(x)\}.
\]
One can verify that $u^*$ is also convex. Moreover, if the function $u$ is also closed (i.e. $\mathrm{epi}~u:=\{(x,\eta)\in\mathbb{R}^n\times\mathbb{R}:u(x)\le\eta\}$ is a closed subset of $\mathbb{R}^{n+1}$), then we have $u^{**}=u$.

For a convex function $u$ on an open set $\Omega \subset \mathbb{R}^n$, we define the subdifferential of $u$ at $x \in \Omega$ by
\[
\partial u(x) := \bigl\{ \xi \in \mathbb{R}^n : u(y) \geq u(x) + \xi\cdot (y - x) ~ \text{for all } y \in \Omega \bigr\}.
\]
For a Borel set $E \subset \Omega$, we define
\[
\partial u(E) = \bigcup_{x\in E} \partial u(x).
\]

The following lemma is a classical result in convex analysis, which will be frequently used in the following sections. For the reader's convenience we provide a short proof.
\begin{lem}\label{Fenchel}
Let $u$ be a closed convex function on $\mathbb{R}^n$, and let $u^*$ denote its Legendre transform. Then
\begin{equation}\label{Fen}
u(x)+u^{*}(\xi)\geq \xi\cdot x\quad \text{for all }x,\xi\in\mathbb{R}^n. 
\end{equation}
Moreover, the following three statements are equivalent:
\begin{itemize}
\item[(a)] $u(x)+u^{*}(\xi)=\xi\cdot x$;
\item[(b)] $\xi\in\partial u(x)$;
\item[(c)] $x\in\partial u^{*}(\xi)$.
\end{itemize}
\end{lem}

\begin{proof}
Inequality \eqref{Fen} holds trivially.
By duality, it suffices to prove the equivalence of (a) and (b).
Suppose $\xi\in\partial u(x)$. Then
\[
u(y)\geq u(x)+\xi\cdot(y-x)\quad \text{for all }y\in\mathbb{R}^n,
\]
which implies
\[
u^{*}(\xi)\leq \xi\cdot x - u(x).
\]
The reverse inequality holds automatically, so equality must hold in \eqref{Fen}.

Now assume $u(x)+u^{*}(\xi)=\xi\cdot x$. For any $y\in\mathbb{R}^n$, we have
\[
\xi\cdot y - u(y)\leq u^{*}(\xi)=\xi\cdot x - u(x),
\]
which implies $\xi\in\partial u(x)$.
\end{proof}

Next we introduce the Monge-Ampère measure, which is related to Legendre transform. For a convex function $u$ on $\Omega$, set 
\[
Mu(E)=|\partial u(E)| \quad \mathrm{for~any~subset~}E\subset \Omega.
\]
Then $Mu:\mathcal{S}\to \mathbb{R}\cup\{+\infty\}$ is a measure, finite on compact sets, where $\mathcal{S}$ is the Borel $\sigma$-algebra defined by
\[
\mathcal{S}=\{E\subset\Omega:\partial u(E) \text{ is Lebesgue measurable}\}.
\]
$Mu$ is called the Monge-Ampère measure associated with the convex function $u$.

The following change of variables formula is from \cite{Le2}, which relates the Monge–Ampère measure to the Legendre transform.
\begin{lem}\label{change variable}
Let $u$ be a convex function on a bounded convex domain $\Omega\subset\mathbb{R}^n$. Then
\begin{equation*}
\int_{\Omega} fd\mu_u = \int_{\partial u(\Omega)} f(Du^*(\xi))d\xi
\quad \text{for all } f \in C(\Omega). 
\end{equation*}
\end{lem}

We need two more technical lemmas about Legendre transform.
\begin{lem}\label{C1}
    Let $\Omega\subset\mathbb{R}^n$ be a convex domain. If $u:\bar\Omega\to\mathbb{R}$ is a closed and strictly convex function, then $v:=u^*\in C^1(\mathbb{R}^n)$.
\end{lem}
\begin{proof}
    Clearly $v$ is a closed convex function with $v^{**}=u$. One can easily verify that
    \[
    \partial v(\xi)=\{x\in\bar\Omega: x \text{ is the maximum point of the function }x\cdot \xi-u(x)\}.
    \]
    Denote the set on the right hand side to be $S$. Suppose there are two different points $x_1,x_2$ in $S$. Define $x_t:=(1-t)x_1+tx_2$. Use the strict convexity of $u$ to compute
    \[
    \xi\cdot x_t-u(x_t)>v(\xi),
    \]
    which contradicts the fact that $x_1$ or $x_2$ is the maximum point of the function $x\mapsto x\cdot \xi-u(x)$. Thus $\partial v(\xi)$ is a singleton. The classical convex analysis tells us a convex function whose subdifferential is a singleton at each point is $C^1$ (see, for example, Theorem 2.38 in \cite{Le3}).
\end{proof}

\begin{lem}\label{strictly convex}
    Let $v:\mathbb{R}^n\to\mathbb{R}$ be a $C^1$ convex function. Then $u:=v^*$ is strictly convex in the interior of $\mathrm{dom}~u$.
\end{lem}
\begin{proof}
    Suppose $u$ is not strictly convex. Then there exists $x_1,x_2\in \mathrm{int~dom}~u$ and $t\in(0,1)$ such that 
    \begin{equation}\label{equal}
    u(x_t)=tu(x_1)+(1-t)u(x_2),\quad \text{where }x_t=tx_1+(1-t)x_2.
    \end{equation}
    Choose $\xi\in\partial u(x_t)$. Since $v$ is $C^1$, we have $x_t=Dv(\xi)$. The definition of subdifferential gives
    \[
    u(x_1)\ge u(x_t)+\xi\cdot(x_1-x_t),\quad u(x_2)\ge u(x_t)+\xi\cdot(x_2-x_t).
    \]
    Combining these two inequalities gives 
    \[
    tu(x_1)+(1-t)u(x_2)\ge u(x_t).
    \]
    Since we have \eqref{equal}, all the inequalities above are equalities. Hence $\xi\in\partial u(x_1)\cap\partial u(x_2)$, which is equivalent to 
    be $x_1,x_2\in\partial v(\xi)$ by Lemma \ref{Fenchel}. Nevertheless, $v\in C^1(\mathbb{R}^n)$ implies $x_1=x_2=Dv(\xi)$, which is a contradiction.
\end{proof}

\subsection{Monge-Ampère equation}
We summarize the results in \cite{Tong-Yau} which will be used in the following sections. Fix $\Omega\subset \mathbb{R}^n$ to be a bounded, smooth and strictly convex domain containing the origin. Let $k$ be a real number with $k+n\ge 0$. Consider the eigenvalue problem
\begin{equation}\label{general MA}
\begin{cases}
(u^\star)^k\det D^2u = \lambda|u|^{n+k} &\text{ in }\Omega\\
u = 0 &\text{ on }\partial\Omega.
\end{cases}
\end{equation}

Tong and Yau proved the following theorem.
\begin{prop}\label{Tong Yau}
    There exists $u\in C^{2,\alpha}(\bar{\Omega})\cap C^\infty(\Omega)$ which is a solution to \eqref{general MA} for $\lambda=\underline{\lambda}_{n+k+1}$, where
    \begin{equation}\label{eigen}
    \underline{\lambda}_{n+k+1}=\inf\left\{ \frac{\int_{\Omega} (-u)(u^\star)^k~d\mu_u}{\int_{\Omega}(-u)^{n+k+1}~dx}: u\in C^{1}(\overline{\Omega})~\text{is strictly convex in }
 \Omega,~u\vert_{\partial\Omega}=0\right\}.
    \end{equation}
    Moreover, if $(\lambda',u')$ is another pair of such solution, then $\lambda=\lambda'$ and there exists $c>0$ such that $u=cu'$.
\end{prop}
The numerator in \eqref{eigen} is called the generalized Monge-Ampère energy, denoted by
\[
\mathcal{E}(u,\Omega):=\int_{\Omega} (-u)(u^\star)^k~d\mu_u.
\]

In fact, one can verify that the eigenfunction $u$ is the minimizer of the functional in the right hand side of \eqref{eigen}.

\subsection{$L_p$ version of Prékopa-Leindler inequality}
The classical Prékopa-Leindler inequality is the functional type of the classical Brunn-Minkowski inequality. Let us recall these inequalities, and readers can find details in \cite{Schneider,Villani}.
\begin{prop}[Brunn-Minkowski inequality]
    Let $K$ and $L$ be nonempty bounded measurable sets in $\mathbb{R}^n$ such that $(1-\lambda)K+\lambda L$ is also measurable for $\lambda\in(0,1)$. Then
    \[
    V_n((1-\lambda)K+\lambda L)^{\frac1n}\ge (1-\lambda)V_n(K)^{\frac1n}+\lambda V_n(L)^{\frac1n}.
    \]
    Here $V_n(\cdot)$ denotes the Lebesgue measure.
\end{prop}
\begin{prop}[Prékopa-Leindler inequality]
    Given $0<\lambda<1$, let $f,g$ and $h$ be nonnegative integrable functions on $\mathbb{R}^n$ satisfying
    \[
    h((1-\lambda)x+\lambda y)\ge f(x)^{1-\lambda}g(y)^\lambda
    \]
    for all $x,y\in \mathbb{R}^n$. Then
    \[
    \int_{\mathbb{R}^n}h~dx\ge (\int_{\mathbb{R}^n}f~dx)^{1-\lambda}(\int_{\mathbb{R}^n}g~dx)^\lambda.
    \]
\end{prop}
The Prékopa-Leindler inequality can quickly imply the Brunn-Minkowksi inequality (see \cite{Gardner}).

Lutwak, Yang and Zhang \cite{Lutwak-Yang-Zhang} established the following $L_p$ Brunn-Minkowski inequality for
compact sets.
\begin{prop}[$L_p$ Brunn-Minkowski inequality]
    For $p\ge 1$ and compact sets $K$ and $L$ in $\mathbb{R}^n$, we have
    \[
    V_n((1-\lambda)K+_p \lambda L)\ge V_n(K)^{1-\lambda}V_n(L)^{\lambda}.
    \]
    for all $\lambda\in[0,1]$.
\end{prop}

Wu \cite{Lutwak-Yang-Zhang} established the following Prékopa-Leindler type inequality which can imply the $L_p$ Brunn-Minkowski inequality.
\begin{prop}[$L_p$ version of the Prékopa-Leindler inequality]\label{Lp PL}
  Let $p\ge 1$, $s,\mu,\omega>0,$ and $f,g,h:\mathbb{R}^n\to [0,\infty)$ be integrable functions with nonepmty supports. If for all $x\in \operatorname{supp}f, y\in\operatorname{supp}g$ and $\lambda\in [0,1]$,
  \begin{equation*}\label{Lp PL condition}
      h((1-\lambda)^{\frac{1}{q}}\mu^\frac{1}{p} x+\lambda^{\frac{1}{q}}\omega^\frac{1}{p} y)^{\frac{1}{s}}\geq (1-\lambda)^{\frac{1}{q}}\mu^\frac{1}{p} f(x)^{\frac{1}{s}}+\lambda^{\frac{1}{q}}\omega^\frac{1}{p} g(y)^{\frac{1}{s}},
      \end{equation*}
  where $q$ is the Hölder conjugate of $p$,
  then
  \begin{equation*}\label{Lp PL conclusion}
      \left(\int_{\mathbb{R}^n}  h~dx\right)^{\frac{p}{n+s}} \geq \mu\left(\int_{\mathbb{R}^n}  f~dx\right)^{\frac{p}{n+s}}+\omega\left(\int_{\mathbb{R}^n}  g~dx\right)^{\frac{p}{n+s}}.
  \end{equation*}
\end{prop}

The $L_p$ version of the Prékopa-Leindler inequality plays an important role when dealing with the denominator in \eqref{eigen}.

\section{An identity of generalized Monge-Ampère energy with respect to the $L_p$ type convolution}\label{Lp type convolution}
Let $\mathcal{K}_s$ be the class of smooth strictly convex compact subsets containing the origin as an interior point in $\mathbb{R}^n$, and $\mathcal{U}_s$ be the class of open subsets whose closure is in $\mathcal{K}_s$. In addition, we define the following class of functions:
\[
\mathcal{C}(\Omega):=\{u:\bar\Omega\to\mathbb{R}: u \in C^1(\bar\Omega)~\text{is strictly convex},\ u|_{\partial\Omega}=0\}
\]
for a domain $\Omega$ in $\mathcal{U}_s$. From now on we fix $k\ge 0$, and denote $p=k+1\ge 1$.

For any $\Omega_0,\Omega_1\in \mathcal{U}_s$ and $u_0\in\mathcal{C}(\Omega_0),u_1\in\mathcal{C}(\Omega_1)$, define the weighted $L_p$ type convolution $u_\alpha$ of $u_0$ and $u_1$ for some $\alpha\in[0,1]$ by giving the Legendre transform of $u_\alpha$:
\[
v_\alpha(\xi):=\left[ (1-\alpha)(u_0^*(\xi))^p+\alpha(u_1^*(\xi))^p\right]^{\frac 1p}.
\]
Then $u_\alpha:=v_\alpha^*=u_\alpha^{**}$. For convenience we also denote $v_0=u_0^*,\ v_1=u_1^*$ (i.e. $v_0,v_1$ are Legendre transforms of $u_0$ and $u_1$ respectively). We remind the readers to distinguish the notation $u^*(\xi)$ and $u^\star(x)$. The former denotes the Legendre transform, and the latter denotes the Legendre transform of $u$ evaluated at the point $Du(x)$. They are different representations of the same quantity in position space and momentum space.

Next we discuss some properties of the $L_p$ type convolution. First we verify $v_\alpha$ is convex, which ensures that its Legendre transform is a well-defined convex function.
\begin{lem}
    $v_\alpha$ is convex in $\mathbb{R}^n$.
\end{lem}
\begin{proof}
    Let $g(a,b)=(a^p+b^p)^{\frac1p}$ for $a,b\ge 0$. The Minkowski inequality implies that
    \[
    g(ta_1+(1-t)a_2,tb_1+(1-t)b_2)\le tg(a_1,b_1)+(1-t)g(a_2,b_2), \quad a_1,a_2,b_1,b_2\ge 0,~t\in[0,1].
    \]
    Set $f_0(\xi)=(1-\alpha)^{\frac 1p}v_0(\xi),\ f_1(\xi)=\alpha^{\frac1p}v_1(\xi)$. Clearly $f_0$ and $f_1$ are both nonnegative and convex. Note that $v_\alpha=g(f_0,f_1)$, which implies immediately that $v_\alpha$ is convex by definition of convex functions.
\end{proof}

The convexity of $v_\alpha$ implies $u_\alpha$, the Legendre transform of $v_\alpha$, is a well-defined convex function in $\mathbb{R}^n$. We focus on the effective domain of $u_\alpha$.
\begin{lem}\label{eff domain}
    Let $\Omega_\alpha$ be the interior of the set $\{x\in\mathbb{R}^n:u_\alpha(x)\le 0\}$. Then
    \[
    \bar\Omega_\alpha=(1-\alpha)\bar\Omega_0+_p \alpha\bar\Omega_1.
    \]
    Equivalently, the support function of $\bar\Omega_\alpha$ is
    \[
    h_\alpha(\xi)=\left[(1-\alpha)h_0(\xi)^p+\alpha h_1(\xi)^p\right]^{\frac1p},
    \]
    where $h_0$ and $h_1$ are support functions of $\bar\Omega_0$ and $\bar\Omega_1$ respectively.

    Furthermore, $u_\alpha=0$ on $\partial \Omega_\alpha$.
\end{lem}
\begin{proof}
    It suffices to show that 
    \[
    \bar\Omega_\alpha=\{x\in\mathbb{R}^n:\xi\cdot x\le h_\alpha(\xi) ~\text{for all }\xi\in\mathbb{R}^n\},
    \]
    where $h_\alpha(\xi)=\left[(1-\alpha)h_0(\xi)^p+\alpha h_1(\xi)^p\right]^{\frac1p}$. We first show the Legendre transform and the support function are equal outside a ball centered at the origin. For $i\in\{0,1\}$, we have
    \[
     v_i(\xi)=\sup_{x\in\bar\Omega_i}\{x\cdot \xi-u_i(x)\}.
    \]
    Since $u_i\in C^1(\bar\Omega_i)$, there exists $R>0$ such that for $|\xi|>R$, the supremum above is attained at some boundary point on $\partial\Omega_i$. Moreover, we have assumed $u_i=0$ on $\partial\Omega_i$, so
    \[
    v_i(\xi)=\sup_{x\in\partial\Omega_i}x\cdot \xi=h_i(\xi).
    \]
    Hence we obtain
    \[
    v_\alpha(\xi)=\left[(1-\alpha)h_0(\xi)^p+\alpha h_1(\xi)^p\right]^{\frac1p}=h_\alpha(\xi)\quad \mathrm{for}\ |\xi|>R.
    \]

    Next we show that $\bar\Omega_\alpha$ is equal to the $L_p$ Minkowski addition. Suppose $x\in\mathbb{R}^n$ satisfies $\xi\cdot x\le h_\alpha(\xi)$ for all $\xi\in\mathbb{R}^n$. Note that $v_\alpha\ge h_\alpha$, so
    \[
    \xi\cdot x-v_\alpha(\xi)\le 0 \quad \text{for all } \xi\in\mathbb{R}^n.
    \]
    Taking the supremum over $\xi$ gives $x\in\bar\Omega_\alpha$.

    Conversely assume $x\in\bar\Omega_\alpha$. Suppose for contradiction that $x\cdot \xi_0>h(\xi_0)$ for some $\xi_0$. Then for all $t>0$ it holds
    \[
    (t\xi_0)\cdot x>th_\alpha(\xi_0)=h_\alpha(t\xi_0).
    \]
    When $t$ is large enough, we have $h_\alpha(t\xi_0)=v_\alpha(t\xi_0)$, so the inequality above turns out to be 
    \[
    t\xi_0\cdot x-v_\alpha(t\xi_0)>0.
    \]
    Letting $t\to\infty$ contradicts the assumption $x\in\bar\Omega_\alpha$.

    Finally we will show $u_\alpha=0$ on $\partial\Omega_\alpha$. Fix $x\in \partial \Omega_\alpha$, it suffices to show that $u_\alpha(x)\ge 0$. We choose $\xi\in\mathbb{S}^{n-1}$ such that $\xi\cdot x=h_\alpha(\xi)$. For sufficiently large $t>R>0$, we have shown that 
    \[
    v_\alpha(t\xi)=th_\alpha(\xi).
    \]
    Now 
    \[
    u_\alpha(x)=\sup_{\eta\in\mathbb{R}^n}\{\eta\cdot x-v_\alpha(\eta)\}\ge \sup_{t>R}\{t\xi\cdot x-th_\alpha(\xi)\}=0.
    \]
\end{proof}

Lemma \ref{eff domain} implies that $\Omega_\alpha=\{u_\alpha<0\}$, $u_\alpha=0$ on $\partial \Omega_\alpha$, and $u_\alpha=\infty$ outside $\bar\Omega_\alpha$. Since $u_0$ and $u_1$ are strictly convex on $\Omega_0$ and $\Omega_1$ respectively, by Lemma \ref{C1} we deduce $v_0,v_1\in C^1(\mathbb{R}^n)$, thus $v_\alpha\in C^1(\mathbb{R}^n)$ by definition. Then Lemma \ref{strictly convex} tells us $u_\alpha$ is also strictly convex. We will analyze the regularity of $u_\alpha$.
\begin{lem}\label{int C1} For each $x_0\in \Omega_\alpha$, there exists a unique $\xi\in\mathbb{R}^n$ such that $\xi\in\partial u_\alpha(x_0)$. Therefore $u_\alpha\in C^1(\Omega_\alpha)$.
\end{lem}
\begin{proof}
    By Lemma \ref{Fenchel}, it suffices to show that the solution of the equation $Dv_\alpha(\xi)=x_0$ is unique. Suppose for contradiction that there exist $\xi_1\ne \xi_2$ with $Dv_\alpha(\xi_1)=Dv_\alpha(\xi_2)=x_0\in\Omega_\alpha$, which implies $v_\alpha$ is an affine function on the segment $[\xi_1,\xi_2]$. By definition $v_\alpha=[(1-\alpha)v_0^p+\alpha v_1^p]^{\frac1p}$, then we deduce $v_0$ and $v_1$ are also affine functions on the segment $[\xi_1,\xi_2]$.

    We claim that $v_i$ is strictly convex in $Du_i(\Omega_i)$ for $i=0,1$. In fact, for $\zeta_1,\zeta_2\in Du_i(\Omega_i)$ with $\zeta_1\ne \zeta_2$, set $y_1=Dv_i(\zeta_1),~ y_2=Dv_i(\zeta_2)$. The strict convexity of $u_i$ implies that the map $Du_i:\Omega_i\to Du_i(\Omega_i)$ is bijective, whose inverse is $Dv_i:Du_i(\Omega_i)\to\Omega_i$. Thus we have $y_1\ne y_2$, and $\zeta_1=Du_i(y_1)~,\zeta_2=Du_i(y_2)$. Therefore,
    \[
    (\zeta_1-\zeta_2)\cdot (Dv_i(\zeta_1)-Dv_i(\zeta_2))=(Du_i(y_1)-Du_i(y_2))\cdot (\zeta_1-\zeta_2)>0,
    \]
    which implies our claim.
    
    Now, since $Dv_\alpha(\xi_1)=x_0\in \Omega_\alpha$, we have $\xi_1=Du_\alpha(x_0)$ by Lemma \ref{Fenchel}. This implies that the function $x\mapsto x\cdot \xi_1-u_\alpha(x)$ attains its maximum in the interior of $\Omega_\alpha$. Hence 
    \[
    v_\alpha(\xi_1)>h_\alpha(\xi_1).
    \]
    By the definition of $v_\alpha$ and $h_\alpha$, we can assume without loss of generality that $v_0(\xi_1)>h_0(\xi_1)$. Therefore, we conclude $\xi_1\in Du_0(\Omega_0)$. This contradicts the strict convexity of $v_0$ in $Du_0(\Omega_0)$.
\end{proof}

By the geometric interpretation of the Legendre transform, we deduce $v_0$ and $v_1$ are $1$-homogeneous functions outside a ball centered at the origin since the gradient of $u_0,u_1$ are bounded. The definition formula of $v_\alpha$ implies that it is also a $1$-homogeneous function outside the ball. Hence $u_\alpha\in C^1(\bar\Omega_\alpha)$. Combining all the preceding discussions, we have proved the $L_p$ type convolution $u_\alpha$ lies in the class $\mathcal{C}(\Omega_\alpha)$.

Next we apply Lemma \ref{change variable} to rewrite the Monge-Ampère energy in another form
\begin{equation}\label{energy}
\int_{\Omega} (-u)(u^\star)^kd\mu_u=\int_{\partial u(\Omega)}-u(Du^*(\xi))(u^*(\xi))^kd\xi \quad \text{for } u\in\mathcal{C}(\Omega).
\end{equation}
Denote
\begin{equation*}
\Phi_j(\xi)=-u_j(Du_j^*(\xi))(u_j^*(\xi))^k, \quad \xi\in\mathbb{R}^n,\ j=0,1,\alpha.
\end{equation*}
The previous lemmas tell us $u_j^*\in C^1(\mathbb{R}^n)$, so $\Phi_j$ is well-defined in the whole space. Set $x=Dv_j(\xi)$, and Lemma \ref{Fenchel} implies
\[
\xi\cdot x=u_j(x)+v_j(\xi),
\]
which is equivalent to say
\[
u_j(Du_j^*(\xi))=\xi\cdot Dv_j(\xi)-v_j(\xi)\quad \text{for all } \xi\in\mathbb{R}^n.
\]

In paticular, when $\xi\notin\partial u_j(\Omega_j)$, we have $\Phi_j(\xi)=0$. In fact, suppose for any $x\in \Omega_j$, it holds $\xi\ne Du_j(x)$, which is equivalent to say $x\ne Dv_j(\xi)$. Note that $y$ is the maximum point of the function $y\mapsto \xi\cdot y-u_j(y)$ if and only if $y\in\partial v(\xi)$. Hence the function $x\mapsto \xi\cdot x-u_j(x)$ attains at its maximum at $\tilde{x}\in\partial\Omega_j$. We have $\tilde{x}=Du_j^*(\xi)$. Therefore we obtain
\[
u_j(Du_j^*(\xi))=u_j(\tilde{x})=0.
\]
\begin{lem}\label{linear}
For any $\xi\in\mathbb{R}^n$, we have
\[
\Phi_\alpha(\xi)=(1-\alpha)\Phi_0(\xi)+\alpha \Phi_1(\xi).
\]
\end{lem}
 \begin{proof}
The identity follows by direct computation. Observe that 
\[
u_j\big(Du_j^{*}(\xi)\big)=\xi\cdot Dv_j(\xi)-v_j(\xi).
\]
Therefore
\[
\Phi_j(\xi)=\big(v_j(\xi)-\xi\cdot Dv_j(\xi)\big)\,v_j(\xi)^{k}
=v_j(\xi)^{p}-\frac{1}{p}\xi\cdot D\big(v_j^{p}\big)(\xi).
\]
Since $v_{\alpha}^{p}=(1-\alpha)v_{0}^{p}+\alpha v_{1}^{p}$, we obtain
\[
\Phi_{\alpha}(\xi)=(1-\alpha)\Phi_{0}(\xi)+\alpha\Phi_{1}(\xi).
\]
\end{proof}

Now we can prove the main theorem in this section, which shows that substituting the $L_p$ convolution into the energy functional $\mathcal{E}$ yields an identity.

\begin{theorem}\label{energy inequality}
For any $\alpha\in[0,1]$, we have
    \[
    \mathcal{E}(u_\alpha,\Omega_\alpha)=(1-\alpha)\mathcal{E}(u_0,\Omega_0)+\alpha\mathcal{E}(u_1,\Omega_1).
    \]
\end{theorem}
\begin{proof}
    By \eqref{energy}, We have
    \begin{align*}
    \mathcal{E}(u_\alpha,\Omega_\alpha)&=\int_{\partial u_\alpha(\Omega_\alpha)}-u_\alpha(Du_\alpha^*(\xi))(u_\alpha^*(\xi))^k~d\xi\\
    &=\int_{\mathbb{R}^n}\Phi_\alpha(\xi)~d\xi\\
    &=\int_{\mathbb{R}^n}[(1-\alpha)\Phi_0(\xi)+\alpha \Phi_1(\xi)]~d\xi\\
    &=(1-\alpha)\int_{\partial u_0(\Omega_0)}\Phi_0(\xi)~d\xi+\alpha\int_{\partial u_1(\Omega_1)}\Phi_1(\xi)~d\xi\\
    &=(1-\alpha)\mathcal{E}(u_0,\Omega_0)+\alpha\mathcal{E}(u_1,\Omega_1).
    \end{align*}
\end{proof}

\section{Proof of Theorem \ref{main result}}\label{proof of main theorem}
We have proved that the numerator in \eqref{general variational character} satisfies an identity with respect to the $L_p$ convolution. In order to deal with the eigenvalue, we still need to use the $L_p$ version of Prékopa-Leindler inequality to establish an inequality for the denominator in \eqref{general variational character}.
\begin{lem}\label{denominator} 
    For any $\alpha\in[0,1]$, 
    \begin{equation}\label{deno}
        \left (\int_{\Omega_\alpha}|u_\alpha|^{n+p}~dx \right )^{\frac p{2n+p}}\ge (1-\alpha)\left ( \int_{\Omega_0}|u_0|^{n+p}~dx \right )^{\frac p{2n+p}}+\alpha\left ( \int_{\Omega_1}|u_1|^{n+p}~dx \right )^{\frac p{2n+p}}.
    \end{equation}
    In paticular, if $u_0,u_1$ satisfy
    \[
    \int_{\Omega_i}|u_i|^{n+p}~dx=1,\quad i=0,1,
    \]
    then
    \[
    \int_{\Omega_\alpha}|u_\alpha|^{n+p}~dx\ge 1.
    \]
\end{lem}
\begin{proof}
    Let $q=\frac{p}{p-1}$ denote the Hölder conjugate of $p$. We claim that for any $x_0\in \Omega_0,x_1\in\Omega_1$ and any $a,b>0$ with 
    \begin{equation}\label{Holder condition}
    \frac{a^q}{(1-\alpha)^{\frac qp}}+\frac{b^q}{\alpha^\frac qp}=1,
    \end{equation}
    we have
    \begin{equation}\label{ab}
    u_\alpha(ax_0+bx_1)\le au_0(x_0)+bu_1(x_1).
    \end{equation}
    In fact, from
    \[
    v_i(\xi)^p\ge \max\{\xi\cdot x_i-u_i(x_i),0\}^p
    \]
    we can deduce
    \[
    v_\alpha(\xi)^p\ge (1-\alpha)\max\{\xi\cdot x_0-u_0(x_0),0\}^p+\alpha\max\{\xi\cdot x_1-u_1(x_1),0\}^p\quad \text{for all }\xi\in\mathbb{R}^n.
    \]
    By \eqref{Holder condition} it is easy to check the following inequality
    \[
    aA+bB\le[(1-\alpha)A^p+\alpha B^p]^{\frac 1p}\quad \text{for all }A,B\ge 0.
    \]
    Set $A=\max\{\xi\cdot x_0-u_0(x_0),0\},~B=\max\{\xi\cdot x_1-u_1(x_1),0\}$, and we obtain
    \begin{align*}
        v_\alpha(\xi)&\ge aA+bB\\
        &\ge a(\xi\cdot x_0-u_0(x_0))+b(\xi\cdot x_1-u_1(x_1))\\
        &=\xi\cdot(ax_0+bx_1)-(au_0(x_0)+bu_1(x_1)).
    \end{align*}
    Thus
    \begin{equation}\label{ab inq}
        \xi\cdot(ax_0+bx_1)-v_\alpha(\xi)\le au_0(x_0)+bu_1(x_1).
    \end{equation}
    Taking the supremum over $\xi$ in \eqref{ab inq}, we obtain the inequality \eqref{ab} in our claim.

    In order to apply the $L_p$ version of Prékopa-Leindler inequality, we substitute the symbols appearing in Proposition \ref{Lp PL} with concrete quantities relevant to our problem:
    \[
    h:=(-u_\alpha)^{n+p},\ f:=(-u_0)^{n+p},\ g:=(-u_1)^{n+p},\ s=n+p,\ \mu=1-\alpha,\ \omega=\alpha.
    \]
    Given $\lambda\in[0,1]$, define
    \[
    a=(1-\lambda)^{\frac1q}(1-\alpha)^{\frac1p},\ b=\lambda^{\frac1q}\alpha^{\frac1p}.
    \]
    Then the condition in Proposition \ref{Lp PL} is satisfied. Hence the inequality \eqref{deno} holds.
\end{proof}

Lemma \ref{ab} immediately implies the eigenvalue satisfies a convexity inequality with respect to $L_p$-Minkowski addition.
\begin{lem}\label{eigen convex ine}
    For $\Omega_0,\Omega_1\in \mathcal{U}_s$ and $\alpha\in[0,1]$, we have
    \begin{equation*}
    \lambda(\Omega_\alpha)\le (1-\alpha)\lambda(\Omega_0)+\alpha\lambda(\Omega_1).
    \end{equation*}
\end{lem}
\begin{proof}
    Let $u_0,u_1$ be eigenfunctions corresponding to $\lambda(\Omega_0)$ and $\lambda(\Omega_1)$ respectively with
    \[
    \int_{\Omega_i}|u_i|^{n+p}~dx=1,\quad i=0,1.
    \]
    It follows from \cite{Tong-Yau} that $u_0\in\mathcal{C}(\Omega_0)$, $u_1\in\mathcal{C}(\Omega_1)$. Thus their $L_p$ convolution $u_\alpha$ is also in $\mathcal{C}(\Omega_\alpha)$. According to the formula \eqref{general variational character} , Theorem \ref{energy inequality} and Lemma \ref{denominator}, we compute that
    \begin{align*}
        \lambda(\Omega_\alpha)&\le \frac{\int_{\Omega_\alpha}(-u_\alpha)(u_\alpha^\star)^k~d\mu_{u_\alpha}}{\int_{\Omega_\alpha}|u_\alpha|^{n+p}~dx}\\
        &\le \int_{\Omega_\alpha}(-u_\alpha)(u_\alpha^\star)^k~d\mu_{u_\alpha}\\
        &=(1-\alpha)\int_{\Omega_0}(-u_0)(u_0^\star)^k~d\mu_{u_0}+\alpha\int_{\Omega_1}(-u_1)(u_1^\star)^k~d\mu_{u_1}\\
        &=(1-\alpha)\lambda(\Omega_0)+\alpha\lambda(\Omega_1).
    \end{align*}
\end{proof}

The following lemma shows that the eigenvalue satisfies the homogeneity with respect to domain dilations.
\begin{lem}\label{dilation}
    If $\Omega\in\mathcal{U}_s$, $t>0$, then $\lambda(t\Omega)=t^{-2n}\lambda(\Omega)$. 
\end{lem}
\begin{proof}
    Let $u$ be the eigenfunction corresponding to the eigenvalue $\lambda(\Omega)$. Then $u$ satisfies 
    \begin{equation*}
      \begin{cases}
      (u^\star)^k\det D^2u = \lambda(\Omega)|u|^{n+k} &\text{ in }\Omega\\
      u = 0 &\text{ on }\partial\Omega.
      \end{cases}
    \end{equation*}
Define $u_t(x):=u(x/t)$. Then $u_t$ satisfies
    \begin{equation*}
      \begin{cases}
      (u_t^\star)^k\det D^2u_t = t^{-2n}\lambda(\Omega)|u_t|^{n+k} &\text{ in }t\Omega\\
      u_t = 0 &\text{ on }\partial(t\Omega).
      \end{cases}
    \end{equation*}
By uniqueness of the eigenvalue, we conclude that $\lambda(t\Omega)=t^{-2n}\lambda(\Omega)$.
\end{proof}

All preliminary steps are now complete. Let us start to prove our main theorem.
\begin{proof}[proof of Theorem \ref{main result}]
    Set $\Omega_0':=\lambda(\Omega_0)^{\frac1{2n}}\Omega_0$ and $\Omega_1':=\lambda(\Omega_1)^{\frac1{2n}}\Omega_1$. By Lemma \ref{dilation}, we have $\lambda(\Omega_0')=\lambda(\Omega_1')=1$. Define $\alpha'\in[0,1]$ to be
    \[
    \alpha'=\frac{\alpha\lambda(\Omega_1)^{-\frac p{2n}}}{(1-\alpha)\lambda(\Omega_0)^{-\frac p{2n}}+\alpha\lambda(\Omega_1)^{-\frac p{2n}}}.
    \]
    In addition, we define the domain
    \[
    \Omega_\alpha':=(1-\alpha')\Omega_0'+_p\alpha'\Omega_1'.
    \]
    Recall that $\Omega_\alpha:=(1-\alpha)\Omega_0+_p\alpha\Omega_1$, and $h_i,h_i'$ denote the support functions of $\Omega_i$ and $\Omega_i'$ respectively for $i=0,1,\alpha$. We compute that
    \begin{align*}
        (h_\alpha')^p&=(1-\alpha')(h_0')^p+\alpha'(h_1')^p\\
        &=(1-\alpha')\lambda(\Omega_0)^{\frac p{2n}}h_0^p+\alpha'\lambda(\Omega_1)^{\frac p{2n}}h_1^p\\
        &=\frac{(1-\alpha)h_0^p+\alpha h_1^p}{(1-\alpha)\lambda(\Omega_0)^{-\frac p{2n}}+\alpha\lambda(\Omega_1)^{-\frac p{2n}}}\\
        &=\frac{h_\alpha^p}{(1-\alpha)\lambda(\Omega_0)^{-\frac p{2n}}+\alpha\lambda(\Omega_1)^{-\frac p{2n}}}.
    \end{align*}
    This identity of support functions implies two convex bodies are equal, that is
    \begin{equation}\label{1}
    \Omega_\alpha'=\left [(1-\alpha)\lambda(\Omega_0)^{-\frac p{2n}}+\alpha\lambda(\Omega_1)^{-\frac p{2n}}\right ]^{-\frac 1p}\Omega_\alpha.
    \end{equation}
    Then Lemma \ref{dilation} and formula \eqref{1} shows that
    \begin{equation}\label{2}
        \lambda(\Omega_\alpha')=\left [ (1-\alpha)\lambda(\Omega_0)^{-\frac p{2n}}+\alpha\lambda(\Omega_1)^{-\frac p{2n}} \right ]^{\frac{2n}p}\lambda(\Omega_\alpha).
    \end{equation}
    From Lemma \ref{eigen convex ine} we have
    \begin{equation}\label{3}
        \lambda(\Omega_\alpha')\le (1-\alpha')\lambda(\Omega_0')+\alpha'\lambda(\Omega_1')=1.
    \end{equation}
    Finally, \eqref{2}\eqref{3} immediately give \eqref{Lp BM}, which finishes the proof.
\end{proof}

\section*{Acknowledgment}
F. Jiang has been supported by the National Natural Science Foundation of China (No. 12271093) and the Jiangsu Provincial Scientific Research Center of Applied Mathematics (Grant No. BK20233002), and Shanghai Institute for Mathematics and Interdisciplinary Sciences (SIMIS) under grant number SIMIS-ID-2025-AD.

\end{document}